\documentclass[a4paper]{amsart}
\usepackage{amsmath,amsthm,amssymb}
\usepackage[pdftex]{graphicx}
\usepackage{mathrsfs}
\usepackage[dvipdfmx, usenames]{color}
\usepackage{ascmac}
\usepackage[all]{xy}

\newtheorem{thm}{Theorem}[section]
\newtheorem{lem}[thm]{Lemma}
\newtheorem{cor}[thm]{Corollary}
\newtheorem{prop}[thm]{Proposition}

\theoremstyle{definition}

\newtheorem{example}[thm]{Example}

\newtheorem{defn}[thm]{Definition}

\theoremstyle{remark}

\newtheorem{rem}[thm]{Remark}        

\numberwithin{equation}{section}

\def\XXint#1#2#3{{\setbox0=\hbox{$#1{#2#3}{\int}$}
\vcenter{\hbox{$#2#3$}}\kern-.5\wd0}}

\newcommand{\R}{\mathbb{R}}

\newcommand{\Lip}{\mathsf{Lip}}
\newcommand{\Cpl}{\mathsf{Cpl}}
\newcommand{\di}{\mathsf{Diam}\,}
\newcommand{\supp}{\mathsf{supp}\,}

\newcommand{\lv}{\left\vert}
\newcommand{\rv}{\right\vert}

\newcommand{\wtilde}[1]{\widetilde{#1}}

\newcommand{\CAT}{\mathsf{CAT}}
\newcommand{\Obs}{\mathsf{ObsDiam}}
\newcommand{\Dis}{\mathsf{dis}\,}

\newcommand{\calB}{\mathcal{B}}
\newcommand{\calC}{\mathcal{C}}

\newcommand{\calH}{\mathcal{H}}

\newcommand{\calK}{\mathcal{K}}
\newcommand{\calL}{\mathcal{L}}

\newcommand{\calN}{\mathcal{N}}

\newcommand{\calP}{\mathcal{P}}

\newcommand{\calX}{\mathcal{X}}
\newcommand{\calY}{\mathcal{Y}}

\newcommand{\bbN}{\mathbb{N}}

\newcommand{\frakY}{\mathfrak{Y}}

\usepackage[abbrev,nobysame]{amsrefs}

\makeatletter
\def\@makefnmark{%
\leavevmode
\raise.9ex\hbox{\check@mathfonts
\fontsize\sf@size\z@\normalfont%
\@thefnmark}%
}
\makeatother

\title{Observable diameters with varying screens}
\author{Yu Kitabeppu}
\author{Naoto Nishida}
\address[Yu Kitabeppu]{Kumamoto University}
\email[Yu Kitabeppu]{ybeppu@kumamoto-u.ac.jp}
\address[Naoto Nishida]{Research and Consulting of Regional Science Co, Ltd.}
\email[Naoto Nishida]{n-nishida@chklab.com}

\begin{document}
\maketitle
 \begin{abstract}
  In this paper, we obtain the limit formula of the observable diameter with non-Euclidean screen. In order to treat a sequence of observable diameters with varying screens, we define new types of observable diameters with errors.  
 \end{abstract}
%
%
\section{Introduction}
The concentration of measure phenomenon is one of the important subject in the geometry of metric measure spaces. Observable diameter, which is introduced by Gromov \cite{Grom}, measures how smaller set possessing large mass. More precisely, for a given mm space $X$, a metric space $Y$ (we call it screen), and a small real number $\kappa>0$, $\Obs_Y(X;-\kappa)$ is defined to be the supremum of the \emph{partial diameter} $\di(f_{\#}\mu_X;1-\kappa)$ where $f:X\rightarrow Y$ runs over all 1-Lipschitz maps, see Definition \ref{def:obs}. It is a fundamental quantity to study the concentration of measures. In the several literatures, the case $Y=\R$ is well studied(see for instance \cite{Shioya} and the reference therein). On the other hand, the observable diameter with generic screen $Y$ is studied for the relation to L\'evy family(\cite{Grom,Fdouble,Flp}). However, as long as the authors known, no one consider the behavior of observable diameters with varying screens. In this paper, we consider the limit formula of a sequence of observable diameters with varying screens. 
\par For the Gromov-Hausdorff convergent sequence of metric spaces $Y_n\xrightarrow{GH} Y$, we are able to construct a family of approximation maps $f_n:Y_n\rightarrow Y$. Such maps can be made to be Borel measurable, but not even continuous in general. Hence we cannot compare the Lipschitz maps from $X$ to $Y_n$ with those to $Y$ via approximation maps. This means, we cannot compare $\Obs_{Y_n}(X;-\kappa)$ with $\Obs_Y(X;-\kappa)$ directly. To overcome the problem, we define new types of observable diameters with errors( see Definition \ref{def:obserrors}). Unfortunately, the observable diameters with errors do not always converges to the observable diameter as errors tending to 0. We therefore give a sufficient condition, named \emph{Lipschitz-approximated family}, for getting the convergence of observable diameters with errors to the original one. Then we obtain the main result as follows. 
\begin{thm}[Theorem \ref{thm:main}]
 Let $X$ be an mm space and $\frakY$ a family of metric spaces being a Lipschitz-approximated family to $X$ (see Definition \ref{def:lipapprox}). 
  \begin{enumerate}
   \item Assume compact metric spaces $Y_n$, $Y\in\frakY$ satisfies $Y_n\xrightarrow{GH}Y$. Then it holds 
   \begin{align}
    \Obs_Y(X;-\kappa)=\lim_{n\rightarrow\infty}\Obs_{Y_n}(X;-\kappa)\notag
   \end{align} 
   for $\calL^1$-a.e. $\kappa\in(0,1)$. 
   \item Assume noncompact metric spaces $Y_n,Y\in\frakY$ satisfying $(Y_n,o_n)\xrightarrow{pGH}(Y,o)$ for some points $o_n\in Y_n$, $o\in Y$. Then it holds 
   \begin{align}
    \Obs_Y(X;-\kappa)\leq \liminf_{n\rightarrow\infty}\Obs_{Y_n}(X;-\kappa)\notag
   \end{align}  
   for $\calL^1$-a.e. $\kappa\in(0,1)$. 
  \end{enumerate} 
\end{thm}
\par The article is organized as follows. In Section 2, we give fundamental notation and facts of geometry of metric measure spaces. In Section 3, we introduce two types of observable diameters with errors. Each of them fits to the box convergence and the Gromov-Hausdorff one respectively. In Section 4, we introduce a concept of Lipschitz-approximated family, which is to connect between Lipschitz maps and almost Lipschitz maps with errors. See Definition \ref{def:lipapprox} for the precise definition. And we prove the main result.     
%
%
\section{Preliminaries}

In this paper, we call a triplet $X=(X,d_X,\mu_X)$ an \emph{mm space} if $(X,d_X)$ is a complete separable metric space with a Borel probability measure $\mu_X$. Two mm spaces $X,Y$ are \emph{isomorphic} to each other if there exists an isometry $f:\supp\mu_X\rightarrow \supp\mu_Y$ such that $f_{\#}\mu_X=\mu_Y$, where $f_{\#}\mu_X$ is the push-forward measure defined as $f_{\#}\mu_X(A):=\mu_X(f^{-1}(A))$ for any Borel set $A\subset Y$. Denote by $\calX$ the set of all mm-isomorphism classes of mm spaces. In this paper, we always assume $X=\supp\mu_X$ for any mm space. A metric space $(X,d)$ is a \emph{geodesic space} if for any two points $x,y\in X$, there exists a continuous curve $\gamma:[0,1]\rightarrow X$ with $\gamma(0)=x$, $\gamma(1)=y$ such that 
\begin{align}
 L(\gamma)&:=\sup\left\{\sum_{i=1}^nd(\gamma(t_i),\gamma(t_{i-1}))\;;\; 0=t_0<t_1<\cdots<t_n=1\right\}\notag\\
 &=d(x,y).\notag
\end{align}
We call such a curve \emph{geodesic}. A geodesic $\gamma:[0,1]\rightarrow X$ is said to be \emph{constant speed} if $d(\gamma(s),\gamma(t))=\lv t-s\rv d(x,y)$ holds for any $s,t\in[0,1]$. 
\smallskip 
\par For a positive number $\varepsilon>0$, we denote the open $\varepsilon$-ball centered at $x$ and the $\varepsilon$-neighborhood of a subset $A$ by 
\begin{align}
 &U_{\varepsilon}(x):=\left\{y\in X\;;\; d_X(y,x)<\varepsilon\right\},\notag\\
 &U_{\varepsilon}(A):=\left\{y\in X\;;\;d_X(y,A)=\inf_{z\in A}d_X(y,z)<\varepsilon\right\}.\notag
\end{align}
Also we denote the closed $\varepsilon$-ball centered at $x$ by $B_{\varepsilon}(x)$. 
\subsection{Distance functions on mm spaces}
In this subsection, we give some concepts of metrics on spaces or functions. See \cite{BBI,Shioya} for more details. Let $(X,d)$ be a complete separable metric space. We denote by $\calP(X)$ and by $\calB(X)$ the set of all Borel probability measures on $X$ and the set of all Borel measurable subsets in $X$, respectively. 
 \begin{defn}[Prokhorov distance]
  For given two probability measures $\mu,\nu\in\calP(X)$, the \emph{Prokhorov distance} $d_P(\mu,\nu)$ between them is defined to be the infimum of $\varepsilon>0$ satisfying 
  \begin{align}
   \mu(U_{\varepsilon}(A))\geq \nu(A)-\varepsilon\notag
  \end{align}
  for any Borel subset $A\subset X$. 
 \end{defn}
The topology induced from Prokhorov metric coincides with the weak topology on $\calP(X)$.  
 \begin{defn}[Ky Fan metric]
  Let $(Z,\mu)$ be a measure space and $Y$ a separable metric space. For two $\mu$-measurable maps $f,g:Z\rightarrow Y$, the \emph{Ky Fan distance} $d_{KF}^{\mu}(f,g)$ between them is defined to be the infimum of $\varepsilon\geq 0$ satisfying 
  \begin{align}
   \mu\left(\left\{z\in Z\;;\;d_Y(f(z),g(z))>\varepsilon\right\}\right)\leq \varepsilon.\notag
  \end{align}
 \end{defn}
$d_{KF}^{\mu}$ is a metric on the set of all $\mu$-equivalent classes of $\mu$-measurable maps from $Z$ to $Y$. 
 \begin{lem}\label{lem:KFP}
  Let $X$ be a topological space with a Borel probability measure $\mu$ and $Y$ a separable metric space. For any two $\mu$-measurable maps $f,g:X\rightarrow Y$, we have 
  \begin{align}
   d_P(f_{\#}\mu,g_{\#}\mu)\leq d_{KF}^{\mu}(f,g).\notag
  \end{align}
 \end{lem}
For two topological spaces $Y,Z$ and $\mu\in\calP(Y)$, $\nu\in\calP(Z)$, a Borel probability measure $\xi\in\calP(Y\times Z)$ is called a \emph{coupling between $\mu$ and $\nu$} if 
\begin{align}
 \left\{\begin{aligned}
  &\xi(A\times Z)=\mu(A)\notag\\
  &\xi(Y\times B)=\nu(B)
 \end{aligned}\right.
\end{align} 
holds for any $A\in\calB(Y)$ and $B\in\calB(Z)$. We denote by $\Cpl(\mu,\nu)$ the set of all couplings between $\mu$ and $\nu$. Note that the product of such measures $\mu\otimes \nu$ is a coupling of them, hence $\Cpl(\mu,\nu)\neq \emptyset$. When $Y,Z$ are mm spaces, we denote $\Cpl(Y,Z):=\Cpl(\mu_Y,\mu_Z)$. Now we are able to define the \emph{box distance}. 
 \begin{defn}[Box distance]
  For given $X,Y\in\calX$, the \emph{box distance} $\Box(X,Y)$ is defined as follows; 
  \begin{align}
   \Box(X,Y):=\inf_{\xi\in\Cpl(X,Y)}\inf_{S\subset \calB(X\times Y)}\max\left\{1-\xi(S),\Dis(S)\right\},\notag
  \end{align}
  where 
  \begin{align}
   \Dis(S):=\sup\left\{\lv d_X(x_1,x_2)-d_Y(y_1,y_2)\rv\;;\; (x_1,y_1), (x_2,y_2)\in S\right\}.\notag
  \end{align}
  The function $\Box$ is a metric on $\calX$. 
 \end{defn}
 \begin{rem}
  The box distance is introduced by Gromov \cite{Grom}. The formulation presented here is due to Nakajima \cite{Nbox}. 
 \end{rem}
We denote by $X_n\xrightarrow{\Box}X$ if $\Box(X_n,X)\rightarrow 0$. 
 \begin{defn}
  Let $X,Y$ be mm spaces and $\varepsilon$ a positive number. A map $f:X\rightarrow Y$ is called an \emph{$\varepsilon$-mm isomorphism} if the following conditions are held; there exists a Borel subset $\wtilde{X}\subset X$ such that 
  \begin{enumerate}
   \item $\mu_X(\wtilde{X})\geq 1-\varepsilon$,
   \item $\lv d_X(x_1,x_2)-d_Y(f(x_1),f(x_2))\rv\leq \varepsilon$ for any $x_1,x_2\in \wtilde X$, 
   \item $d_P(f_{\#}\mu_X,\mu_Y)\leq \varepsilon$. 
  \end{enumerate} 
  We call the set $\wtilde{X}$ the \emph{non-exceptional set of $f$}.  
 \end{defn}
 \begin{prop}
  Let $X,Y$ be mm spaces. The following holds. 
  \begin{enumerate}
   \item If there exists an $\varepsilon$-mm isomorphism $f:X\rightarrow Y$, then $\Box(X,Y)\leq 3\varepsilon$. 
   \item If $\Box(X,Y)<\varepsilon$, then there exists a $3\varepsilon$-mm isomorphism between $X$ and $Y$. 
  \end{enumerate}
 \end{prop}
We also know another type of metric on compact mm spaces, namely the \emph{Gromov-Hausdorff distance}. 
 \begin{defn}[Hausdorff distance]
  Let $(Z,\rho)$ be a metric space.  
  Then the Hausdorff metric $H\!\rho$ on $2^Z$ is defined by 
  \begin{align}
   H\!\rho(A,B):=\inf\left\{\varepsilon>0\;;\;U_{\varepsilon}(A)\supset B,\;U_{\varepsilon}(B)\supset A\right\}.\notag
  \end{align}
  $H\!\rho$ is a distance function on $\calK(Z)$ that is the set of all compact subsets in $Z$.
 \end{defn}
 \begin{defn}[Gromov-Hausdorff distance]
  Let $\calC$ be a set of all isometric classes of compact metric spaces. The \emph{Gromov-Hausdorff distance} $d_{GH}$ on $\calC$ is defined by  
  \begin{align}
   d_{GH}(A,B):=\inf_{Z,\varphi,\psi}H\!d_Z(\varphi(A),\psi(B))\notag
  \end{align}
  where $Z,\varphi,\psi$ run over all which $(Z,d_Z)$ is a separable metric space, $\varphi:A\rightarrow Z$, $\psi:B\rightarrow Z$ are isometric embeddings. 
 \end{defn}
We are able to translate the Gromov-Hausdorff distance into the existence of approximation maps. 
 \begin{defn}
  Let $Y,Z$ be compact metric spaces and $\varepsilon$ a positive number. A Borel measurable map $f:Y\rightarrow Z$ is called an \emph{$\varepsilon$-approximation map} if 
  \begin{enumerate}
   \item $\lv d_Y(y_1,y_2)-d_Z(f(y_1),f(y_2))\rv<\varepsilon$, 
   \item $U_{\varepsilon}(f(Y))\supset Z$. 
  \end{enumerate}
 \end{defn}
 \begin{prop}
  Let $X,Y$ be compact metric spaces with $d_{GH}(X,Y)<\varepsilon$, then there exists a $3\varepsilon$-approximation map $f:X\rightarrow Y$. Conversely, if there exists an $\varepsilon$-approximation map from $X$ to $Y$, then it holds $d_{GH}(X,Y)<3\varepsilon$. 
 \end{prop}
 \begin{defn}[Measured Gromov-Hausdorff convergence]
  Let $X_n,X$ be compact mm spaces. Then $X_n$ converges to $X$ in the \emph{measured Gromov-Hausdorff sense}, we write it by $X_n\xrightarrow{mGH}X$, if there exists a decreasing sequence of positive numbers $\varepsilon_n\downarrow0$, $\varepsilon_n$-approximation maps $f_n:X_n\rightarrow X$ such that $(f_n)_{\#}\mu_{X_n}\rightarrow \mu_X$ weakly.  
 \end{defn}
 \begin{rem}
  As aforementioned, we always assume $\supp\mu_X=X$ for any mm spaces. Hence the above definition makes sense. 
 \end{rem}
 Let $Y$ be a metric space and $o\in Y$. We call a pair $(Y,o)$ a pointed metric space. Finally we are able to define the \emph{pointed Gromov-Hausdorff convergence} 
 \begin{defn}
  Let $(Y_n,o_n)$, $(Y,o)$ be proper pointed metric spaces. We say that $(Y_n,o_n)$ converges to $(Y,o)$ in the pointed Gromov-Hausdorff sense, write it by $(Y_n,o_n)\xrightarrow{pGH}(Y,o)$, if there exist increasing numbers $R_n\uparrow\infty$, decreasing numbers $\varepsilon_n\downarrow0$, sequence of Borel measurable maps $f_n:B_{R_n}(o)\rightarrow Y_n$ such that it holds 
  \begin{enumerate}
   \item $\lv d_{Y_n}(f_n(z_1),f_n(z_2))-d_Y(z_1,z_2)\rv<\varepsilon_n$ for any $z_1,z_2\in B_{R_n}(o)$, 
   \item $U_{\varepsilon_n}(f_n(B_{R_n}(o)))\supset B_{R_n-\varepsilon_n}(o_n)$. 
  \end{enumerate}
 \end{defn}
 \begin{rem}
  In the case where all $Y_n, Y$ are geodesic spaces, then the condition (2) can be replaced by 
  \begin{enumerate}
   \item[(2)'] $U_{\varepsilon_n}(f_n(B_{R_n}(o)))\supset B_{R_n}(o_n)$. 
  \end{enumerate}
 \end{rem}
 \begin{rem}
  It is the standard way to define the pointed Gromov-Hausdorff convergence by using maps from $B_{R_n}(o_n)$ to $Y$ instead of $f:B_{R_n}(o)\rightarrow Y_n$. Both definitions are equivalent. 
 \end{rem}
\subsection{Observable diameter}
 \begin{defn}[Partial diameter]
  Let $Z$ be a metric space and $\kappa>0$. For a probability measure $\mu\in\calP(Z)$, we define the \emph{partial diameter of $\mu$} by 
  \begin{align}
   \di(\mu;1-\kappa):=\inf\left\{\di (A)\;;\; A\in\calB(Z),\;\mu(A)\geq 1-\kappa\right\}.\notag
  \end{align}
 \end{defn}
The set of all 1-Lipschitz maps from $X$ to $Y$ is denoted by $\Lip_1(X,Y)$. 
 \begin{defn}[Observable diameter]\label{def:obs}
  Let $X$ be an mm space and $Y$ a metric space. For $\kappa>0$, the \emph{observable diameter of $X$ with screen $Y$} is defined by 
  \begin{align}
   \Obs_Y(X;-\kappa):=\sup\left\{\di(f_{\#}\mu_X;1-\kappa)\;;\; f\in \Lip_1(X,Y)\right\}.\notag
  \end{align} 
 \end{defn}
By the same argument in \cite{OSh}, we have the following. 
 \begin{prop}[\cite{OSh}]
  The map $\kappa\mapsto \Obs_Y(X;-\kappa)$ is right-continuous. 
 \end{prop}
 \begin{rem}
  For any fixed $\kappa\in(0,1)$ and any normed space $Y$, $X\mapsto \Obs_Y(X;-\kappa)$ is never continuous with respect to neither the $\Box$-topology nor the concentration topology \cite{KNS}*{Corollary 3.9}. 
 \end{rem}
%
%
\section{Observable diameter with errors}
In this section, we define two types of observable diameters with errors, which are modified versions of the observable diameter. In order to define that, we need the following definition. 
 \begin{defn}
  Let $X$ be an mm space and $Y$ a metric space. For given $\delta>0$, a Borel measurable map $f:X\rightarrow Y$ is called an \emph{almost 1-Lipschitz map with an error $\delta$} if there exists a Borel subset $\wtilde{X}\subset X$ with $\mu_{X}(\wtilde{X})\geq 1-\delta$ such that 
  \begin{align}
   d_Y(f(z_1),f(z_2))\leq d_X(z_1,z_2)+\delta\notag
  \end{align}
  holds for any $z_1,z_2\in \wtilde{X}$. We call $\tilde X$ the non-exceptional set. 
  \par A Borel measurable map $g:X\rightarrow Y$ is called a \emph{1-Lipschitz map with an error $\delta$} if 
  \begin{align}
   d_Y(g(z_1),g(z_2))\leq d_X(z_1,z_2)+\delta\notag
  \end{align}
  holds for any $z_1,z_2\in X$. We denote 
  \begin{align} 
   &\Lip_1^{\delta}(X,Y):=\left\{f:X\rightarrow Y\;;\;\text{ $f$ is a 1-Lipschitz map with error $\delta$}\right\},\notag\\
   &\wtilde{\Lip_1^{\delta}}(X,Y):=\left\{f:X\rightarrow Y\;;\;\text{ $f$ is an almost 1-Lipschitz map with error $\delta$}\right\}.\notag
  \end{align}
 \end{defn}
 \begin{rem}
  Almost 1-Lipschitz map with an error $\delta$ is called \emph{1-Lipschitz up to $\delta$} or \emph{1-Lipschitz up to an additive error $\delta$} \cite{Grom,Shioya}. 
 \end{rem}
It is clear that $\Lip_1\subset \Lip_1^{\delta}\subset \wtilde{\Lip_1^{\delta}}$. It is an important remark that each $\delta$-mm isomorphism (resp. $\delta$-approximation map) $f:X\rightarrow Y$ belongs to $\wtilde{\Lip_1^{\delta}}(X,Y)$ (resp. $\Lip_1^{\delta}(X,Y)$). We define the observable diameter with error. 
 \begin{defn}\label{def:obserrors}
  Let $X$ be an mm space and $Y$ a metric space. For given $\delta>0$ and $\kappa>0$, we define 
  \begin{align}
   &\Obs_Y^{\delta}(X;-\kappa):=\sup\left\{\di(f_{\#}\mu_X;1-\kappa)\;;\;f\in \Lip_1^{\delta}(X,Y)\right\},\notag\\
   &\wtilde{\Obs_Y^{\delta}}(X;-\kappa):=\sup\left\{\di(f_{\#}\mu_X;1-\kappa)\;;\;f\in \wtilde{\Lip_1^{\delta}}(X,Y)\right\}.\notag
  \end{align}
  Since $\delta\mapsto \Obs_Y^{\delta},\;\delta\mapsto\wtilde{\Obs_Y^{\delta}}$ are nondecreasing and nonnegative, we are able to define the followings; 
  \begin{align}
   &\Obs_Y^+(X;-\kappa):=\lim_{\delta\downarrow 0}\Obs_Y^{\delta}(X;-\kappa),\notag\\
   &\wtilde{\Obs_Y^+}(X;-\kappa):=\lim_{\delta\downarrow 0}\wtilde{\Obs_Y^{\delta}}(X;-\kappa).\notag
  \end{align}
 \end{defn}
\subsection{Observable diameter and box topology}
First we give a proposition that is a relation of the observable diameter and the box topology. 
 \begin{prop}\label{prop:arbbox}
  Let $X_n,X$ be mm spaces and $Y$ a metric space. Assume $X_n\xrightarrow{\Box}X$. Then we have 
  \begin{align}
   &\Obs_Y(X;-\kappa)\leq \lim_{\varepsilon\rightarrow+0}\liminf_{n\rightarrow\infty}\wtilde{\Obs_Y^{\varepsilon}}(X_n;-\kappa),\notag\\
   &\wtilde{\Obs_Y^+}(X;-\kappa)\geq \lim_{\varepsilon\rightarrow+0}\limsup_{n\rightarrow\infty}\Obs_Y(X_n;-(\kappa+\varepsilon))\notag
  \end{align}
 \end{prop}
 \begin{proof}
  Take a positive number $\varepsilon>0$ and fix it. By the assumption, there exist sequences of positive numbers $\varepsilon_n\downarrow0$, $\varepsilon_n$-mm isomorphisms $f_n:X_n\rightarrow X$. For sufficiently large $n$, it holds $\varepsilon_n<\varepsilon$. From now on, we assume $\varepsilon_n<\varepsilon$. Take any $F\in\Lip_1(X,Y)$. Then a map $G:=F\circ f_n: X_n\rightarrow Y$ is an almost 1-Lipschitz map with error $\varepsilon_n$. Note that $d_P({f_n}_{\#}\mu_{X_n},\mu_X)\leq \varepsilon_n$ by the definition of mm-isomorphism $f_n$. 
  \par Let $A\subset Y$ be a Borel subset with $G_{\#}\mu_{X_n}(A)\geq 1-\kappa$. Since $F$ is 1-Lipschitz, $U_{\varepsilon}(F^{-1}(A))\subset F^{-1}(U_{\varepsilon}(A))$ holds. Then we have 
  \begin{align}
   F_{\#}\mu_X(U_{\varepsilon_n}(A))&=\mu_X(F^{-1}(U_{\varepsilon_n}(A)))\geq \mu_X(U_{\varepsilon_n}(F^{-1}(A)))\notag\\
   &\geq (f_n)_{\#}\mu_{X_n}(F^{-1}(A))-\varepsilon_n\notag\\
   &=G_{\#}\mu_{X_n}(A)-\varepsilon_n\geq 1-\kappa-\varepsilon_n.\notag
  \end{align}
  The inequality $\di(U_{\varepsilon_n}(A))\leq \di(A)+2\varepsilon_n$ implies, 
  \begin{align}
   \di(F_{\#}\mu_X;1-\kappa-\varepsilon_n)\leq \di(U_{\varepsilon_n}(A))\leq \di(A)+2\varepsilon_n. \notag
  \end{align}
  Since $A$ is arbitrary, we have 
  \begin{align}
   \di(F_{\#}\mu_X;1-\kappa-\varepsilon_n)&\leq \di(G_{\#}\mu_{X_n};1-\kappa)+2\varepsilon_n\notag\\
   &\leq \wtilde{\Obs_Y^{\varepsilon_n}}(X_n;-\kappa)+2\varepsilon_n.\notag
  \end{align}
  Also, since $F\in \Lip_1(X,Y)$ is arbitrary, we obtain 
  \begin{align}
   \Obs_Y(X;-(\kappa+\varepsilon_n))&\leq \wtilde{\Obs_Y^{\varepsilon_n}}(X_n;-\kappa)+2\varepsilon_n\notag\\
   &\leq \wtilde{\Obs_Y^{\varepsilon}}(X_n;-\kappa)+2\varepsilon_n.\notag
  \end{align}
  From the right continuity as a function of $\kappa$ of $\Obs_Y(X;-\kappa)$, we have 
  \begin{align}
   \Obs_Y(X;-\kappa)\leq \liminf_{n\rightarrow\infty}\wtilde{\Obs_Y^{\varepsilon}}(X_n;-\kappa).\notag
  \end{align} 
  Since $\varepsilon>0$ is arbitrary, we conclude. 
  \smallskip
  \par Instead of $\varepsilon_n$-mm isomorphisms $f_n:X_n\rightarrow X$, using $\varepsilon_n$-mm isomorphisms $g_n:X\rightarrow X_n$ leads the result by the same argument above. 
 \end{proof}  
 \begin{rem}
  It is known that if $X_n\xrightarrow{\Box}X$, then 
  \begin{align}
   \Obs_{\R}(X;-\kappa)&=\lim_{\varepsilon\rightarrow+0}\liminf_{n\rightarrow\infty}\Obs_{\R}(X_n;-(\kappa+\varepsilon))\notag\\
   &=\lim_{\varepsilon\rightarrow+0}\limsup_{n\rightarrow\infty}\Obs_{\R}(X_n;-(\kappa+\varepsilon))\notag
  \end{align}
  holds \cite{OSh}. 
 \end{rem}
 By a similar argument, we obtain the following corollary. 
 \begin{cor}
  Let $X_n,X$ be compact mm spaces and $Y$ a metric space. Assume $X_n\xrightarrow{mGH}X$. Then we have 
  \begin{align}
   &\Obs_Y(X;-\kappa)\leq \lim_{\varepsilon\rightarrow+0}\liminf_{n\rightarrow\infty}\Obs_Y^{\varepsilon}(X_n;-\kappa),\notag\\
   &\Obs_Y^+(X;-\kappa)\geq \lim_{\varepsilon\rightarrow+0}\limsup_{n\rightarrow\infty}\Obs_Y(X_n;-(\kappa+\varepsilon)).\notag
  \end{align}
 \end{cor}
\subsection{Observable diameter and varying screens}
Let $(Y,o)$ be a pointed metric space. For any $R>0$, we define 
\begin{align}
 \Lip_1(X,Y;R):=\left\{f\in \Lip_1(X,Y)\;;\;\sup_{x\in X}d_Y(o,f(x))<R\right\}, \notag
\end{align}
and 
\begin{align}
 \Obs_Y(X;-\kappa;R):=\sup\left\{\di(F_{\#}\mu_X;1-\kappa)\;;\;F\in\Lip_1(X,Y;R)\right\}.\notag
\end{align}
Since $\Lip_1(X,Y)=\cup_{R>0}\Lip_1(X,Y;R)$ for an mm space $X$ of finite diameter, we have $\lim_{R\rightarrow\infty}\Obs_Y(X;-\kappa;R)=\Obs_Y(X;-\kappa)$.
 \begin{thm}\label{thm:pGHlsc}
  Let $X$ be an mm space of finite diameter and $(Y_n,o_n)$, $(Y,o)$ proper pointed metric spaces such that $Y_n\xrightarrow{pGH}Y$. Then 
  \begin{align}
   \Obs_Y(X;-\kappa)\leq \lim_{\varepsilon\rightarrow+0}\liminf_{n\rightarrow\infty}\Obs_{Y_n}^{\varepsilon}(X;-\kappa)\notag
  \end{align} 
  holds. 
 \end{thm}
 \begin{proof}  
  Take any positive number $\varepsilon>0$ and fix it. Then there exists a map $F\in\Lip_1(X,Y;R)$ such that 
  \begin{align}
   \di(F_{\#}\mu_X;1-\kappa)\geq \Obs_Y(X;-\kappa)-\varepsilon.\notag
  \end{align}
  By the assumption, there exist sequences of numbers $R_n\uparrow+\infty$, $\varepsilon_n\downarrow0$, of maps $f_n:B^Y_{R_n}(o)\rightarrow Y_n$ such that 
  \begin{align}
   \left\{\begin{aligned}
    &\lv d_Y(y_1,y_2)-d_{Y_n}(f_n(y_1),f_n(y_2))\rv<\varepsilon_n\notag\\
    &U_{\varepsilon_n}(f_n(B^Y_{R_n}(o)))\supset B^{Y_n}_{R_n-\varepsilon_n}(o_n)
   \end{aligned}\right.
  \end{align}
  Define the map $G_n:=f_n\circ F$. Then $G_n\in \Lip_1^{\varepsilon_n}(X,Y_n)$ for sufficiently large $n$. In the following we assume $R_n>R$ and $\varepsilon_n<\varepsilon$. 
  \par Take any Borel subset $A\in \calB(Y_n)$ with $(G_n)_{\#}\mu_X(A)\geq 1-\kappa$. Then by the definition, we have 
  \begin{align}
   F_{\#}\mu_X(f_n^{-1}(A))=(G_n){\#}\mu_X(A)\geq 1-\kappa.\notag
  \end{align}
  Since $\di (f_n^{-1}(A))\leq \di (A)+4\varepsilon_n$, by a similar argument in Proposition \ref{prop:arbbox}, it holds 
  \begin{align}
   \Obs_Y(X;-\kappa)-\varepsilon&\leq \di(F_{\#}\mu_X;1-\kappa)\leq \Obs_{Y_n}^{\varepsilon_n}(X;-\kappa)+4\varepsilon_n\notag\\
   &\leq \Obs_{Y_n}^{\varepsilon}(X;-\kappa)+4\varepsilon.\label{eq:pGHlsc}
  \end{align}
  Hence we have 
  \begin{align}
   \Obs_Y(X;-\kappa)\leq \liminf_{n\rightarrow \infty}\Obs_{Y_n}^{\varepsilon}(X;-\kappa)+5\varepsilon\notag
  \end{align}
  for arbitrary $\varepsilon>0$. Thus we have the conclusion. 
 \end{proof}
 \begin{rem}
  The inequality in Theorem \ref{thm:pGHlsc} cannot be changed into the equality in general. Let us consider the following example: Define an mm space $X$ by 
  \begin{align}
   &X=[0,1]\cup\{\infty\},\notag\\
   &d_X(x,\infty)=1\text{ for any }x\in[0,1] \text{ and } d_X(x,y)=\lv x-y\rv \text{ for any $x,y\in[0,1]$},\notag\\
   &\mu_X=\frac{1}{2}\delta_{\infty}+\frac{1}{2}\calL^1\vert_{[0,1]}.\notag
  \end{align} 
  And a set $Y$ is defined by $Y:=\R_{\leq 0}\cup X_{/\sim}$ where $x\sim y$ if either $x=0\in\R_{\leq0}$ and $y=0\in X$ or $x=0\in X$ and $y=0\in\R_{\leq 0}$. We define a distance function on $Y$ by 
  \begin{align}
   d_Y(x,y):=\left\{\begin{aligned}
    &\lv x-y\rv\quad\text{if }x,y\in\R_{\leq 0},\\
    &d_X(x,y)\quad\text{if }x,y\in X,\\
    &\lv x\rv+\lv y\rv\quad \text{either }x\in\R_{\leq 0},\;y\in X\setminus\{\infty\}, \text{or }x\in X\setminus \{\infty\},\;y\in\R_{\leq 0},\\
    &\lv x\rv+1\quad\text{if }x\in\R_{\leq0},\;y=\infty,\\
    &\lv y\rv+1\quad\text{if }x=\infty,\;y\in\R_{\leq0}.\notag
   \end{aligned}\right.
  \end{align}
  Take a positive number $0<\kappa<1/2$. Then since $X$ can be isometric embedded into $Y$, it holds 
  \begin{align}
  \Obs_Y(X;-\kappa)\geq \di(\mu_X;1-\kappa)\geq 1.\notag
  \end{align} 
  On the other hand, take a point $-1\in\R_{\leq0}\subset Y$, and consider the partially scaled pointed metric spaces $(Y_n,-1):=(Y,d_{Y_n},-1)$, where 
  \begin{align}
   d_{Y_n}(x,y):=\left\{\begin{aligned}
    &n\lv x-y\rv\quad\text{if }x,y\in\R_{\leq 0},\\
    &d_X(x,y)\quad\text{if }x,y\in X,\\
    &n\lv x\rv+\lv y\rv\quad \text{if }x\in\R_{\leq 0},\;y\in X\setminus\{\infty\},\\
    &\lv x\rv+n\lv y\rv\quad \text{if }x\in X\setminus\{\infty\},\;y\in\R_{\leq 0},\\
    &n\lv x\rv+1\quad\text{if }x\in\R_{\leq0},\;y=\infty,\\
    &n\lv y\rv+1\quad\text{if }x=\infty,\;y\in\R_{\leq0}.\notag
   \end{aligned}\right.
  \end{align}
  Then by an elementary argument, we obtain $(Y_n,-1)\xrightarrow{pGH}(\R,0)$. Since any Lipschitz map $f:X\rightarrow \R$ has to maps a point $\{\infty\}$ onto $f(X\setminus \{\infty\})$, it holds $\Obs_{\R}(X;-\kappa)=2^{-1}(1/2-\kappa)$. However $\Obs_{Y_n}(X;-\kappa)\geq 1$. Thus the equality never hold since $\Obs_Y\leq \Obs_Y^{\varepsilon}$ for any $\varepsilon>0$.   
 \end{rem}
For any sequence of compact mm spaces satisfying $Y_n\xrightarrow{GH}Y$, the supremum of diameters of $Y_n$ is finite. Hence the same argument in Theorem \ref{thm:pGHlsc} leads the following corollary. 
 \begin{cor}\label{cor:GHlscusc}
  Let $Y_n, Y$ be compact metric space such that $Y_n\xrightarrow{GH}Y$. Then 
  \begin{align}
   &\Obs_Y(X;-\kappa)\leq \lim_{\varepsilon\rightarrow+0}\liminf_{n\rightarrow\infty}\Obs_{Y_n}^{\varepsilon}(X;-\kappa)\notag\\
   &\Obs_Y^+(X;-\kappa)\geq \limsup_{n\rightarrow\infty}\Obs_{Y_n}(X;-\kappa).\notag
  \end{align} 
 \end{cor}
%
%
\section{Lipschitz-approximated family and observable diameter}
\begin{defn}\label{def:lipapprox}
 For a given mm space $X$, a family of metric spaces $\frakY$ is called a \emph{Lipschitz-approximated family to $X$} if it holds 
 \begin{align}
  \lim_{\delta\rightarrow +0}\sup_{Y\in\frakY}H\!d_{KF}^{\mu_X}(\Lip_1(X,Y),\wtilde{\Lip_1^{\delta}}(X,Y))=0.\label{eq:lipapprox}
 \end{align}
  When a family of metric spaces $\frakY$ consists of only one metric space $Y$, then $Y$ is said to be Lipschitz-approximated to $X$. 
\end{defn}
 \begin{rem}
  It is obvious that the condition (\ref{eq:lipapprox}) is not always held. For example, consider the two points mm space defined as $X=\{x_0,x_1\}$, $d_X(x_0,x_1)=1$, $\mu_X=2^{-1}\delta_{x_0}+2^{-1}\delta_{x_1}$ and the countable metric space $Y=\{y_i\}_{i=0}^{\infty}$ with the metric $d_Y$ on $Y$ defined by $d_Y(y_i,y_i)=0$, ($i=0,1,2,\cdots$), 
  \begin{align}
   d_Y(y_0,y_i)=1+i^{-1}\;(i\neq 0), \;d_Y(y_j,y_k)=2+j^{-1}+k^{-1}\;(j,k\neq 0,\,j\neq k).\notag
  \end{align} 
  Then $\Lip_1(X,Y)=\{\text{constant maps}\}$ while a map $f:X\rightarrow Y$ defined by 
  \begin{align}
    f(x_0):=y_0,\;f(x_1):=y_i\notag
  \end{align}
  belongs to $\wtilde{\Lip_1^{\delta}}(X,Y)$ for $\delta>i^{-1}$. Then $H\!d_{KF}^{\mu_X}(\Lip_1(X,Y),\wtilde{\Lip_1^{\delta}}(X,Y))\geq 2^{-1}$ for any $\delta>0$. 
 \end{rem}
 \begin{thm}\label{thm:coin}
  Let $X$ be an mm space. Assume a metric space $Y$ is Lipschitz-approximated to $X$. Then for $\calL^1$-a.e. $\kappa\in(0,1)$, it holds 
  \begin{align}
   \wtilde{\Obs_Y^+}(X;-\kappa)=\Obs_Y(X;-\kappa).\notag
  \end{align} 
 \end{thm}
 \begin{proof}
  Let $\varepsilon>0$ be an arbitrary positive number less than $\kappa$. By the assumption, 
   \begin{align}
    H\!d_{KF}^{\mu_X}(\Lip_1(X,Y),\wtilde{\Lip_1^{\delta}}(X,Y))<\varepsilon\notag
   \end{align} 
   for small $\delta>0$. Take any $f\in \wtilde{\Lip_1^{\delta}}(X,Y)$. Then we are able to find $g\in \Lip_1(X,Y)$ such that $d_{KF}^{\mu_X}(f,g)<\varepsilon$. Thus by Lemma \ref{lem:KFP}, we have $d_P(f_{\#}\mu_X,g_{\#}\mu_X)<\varepsilon$. 
   \par Take any Borel set $A\in\calB(Y)$ with $g_{\#}\mu_X(A)\geq 1-(\kappa-\varepsilon)$. Then 
   \begin{align}
    f_{\#}\mu_X(U_{\varepsilon}(A))\geq g_{\#}\mu_X(A)-\varepsilon\geq 1-\kappa.\notag
   \end{align}
   Since $\di(U_{\varepsilon}(A))\leq \di(A)+2\varepsilon$, by a similar argument in Proposition \ref{prop:arbbox}, we obtain 
   \begin{align}
    &\wtilde{\Obs_Y^+}(X;-\kappa)\notag\\
    &\leq \wtilde{\Obs_Y^{\delta}}(X;-\kappa)\leq \Obs_Y(X;-(\kappa-\varepsilon))+2\varepsilon.\label{eq:compdelta}
   \end{align}
   Since $\kappa\mapsto\Obs_Y(X;-\kappa)$ is right-continuous and monotone non-increasing, the function is continuous except at most countably many numbers. Thus taking the limit $\varepsilon\rightarrow+0$, it holds 
   \begin{align}
    \wtilde{\Obs_Y^+}(X;-\kappa)\leq \Obs_Y(X;-\kappa)\notag
   \end{align}
   at any continuous point $\kappa$. It is trivial the converse inequality is true. 
 \end{proof}
We are able to rewrite Theorem \ref{thm:pGHlsc} and Corollary \ref{cor:GHlscusc} into simpler statements. 
 \begin{thm}\label{thm:main}
  Let $X$ be an mm space and $\frakY$ a family of metric spaces being a Lipschitz-approximated family to $X$. 
  \begin{enumerate}
   \item Assume compact metric spaces $Y_n$, $Y\in\frakY$ satisfies $Y_n\xrightarrow{GH}Y$. Then it holds 
   \begin{align}
    \Obs_Y(X;-\kappa)=\lim_{n\rightarrow\infty}\Obs_{Y_n}(X;-\kappa)\notag
   \end{align} 
   for $\calL^1$-a.e. $\kappa\in(0,1)$. 
   \item Assume noncompact metric spaces $Y_n,Y\in\frakY$ satisfying $(Y_n,o_n)\xrightarrow{pGH}(Y,o)$ for some points $o_n\in Y_n$, $o\in Y$. Then it holds 
   \begin{align}
    \Obs_Y(X;-\kappa)\leq \liminf_{n\rightarrow\infty}\Obs_{Y_n}(X;-\kappa)\notag
   \end{align}  
   for $\calL^1$-a.e. $\kappa\in(0,1)$. 
  \end{enumerate} 
 \end{thm}
 \begin{proof}
  Both proofs are similar. Let us start proving (1). By the assumption, we have $\eta_n$-approximation maps $Y\rightarrow Y_n$ with $\eta_n\rightarrow+0$. By Theorem \ref{thm:coin}, for $\calL^1$-a.e. $\kappa\in (0,1)$ it holds 
  \begin{align}
   &\Obs_{Y}^{+}(X;-\kappa)=\Obs_{Y}(X;-\kappa),\notag\\
   &\Obs_{Y_n}^{+}(X;-\kappa)=\Obs_{Y_n}(X;-\kappa)\notag   
  \end{align}
  for any $n\in\bbN$. Take such a $\kappa\in(0,1)$ and fix it. Moreover we may assume the continuity of functions $\Obs_Y(X;-(\cdot))$, $\Obs_{Y_n}(X;-(\cdot))$ at $\kappa$. Then by Corollary \ref{cor:GHlscusc}, we have 
  \begin{align}
   \limsup_{n\rightarrow\infty}\Obs_{Y_n}(X;-\kappa)\leq \Obs_Y(X;-\kappa).\notag
  \end{align}
  Take a decreasing sequence $\varepsilon_j\downarrow 0$ so that $\kappa\mapsto \Obs_{Y_n}(X;-\kappa)$, $\Obs_Y(X;-\kappa)$ are continuous at $\kappa+\varepsilon_j$ for any $j$. By (\ref{eq:lipapprox}) and (\ref{eq:compdelta}), for such $\varepsilon_j$, there exists a positive constant $\delta_j$ such that it holds 
  \begin{align}
   \Obs_{Y_n}^{\delta_j}(X;-(\kappa+\varepsilon_j))\leq \Obs_{Y_n}(X;-\kappa)+2\varepsilon_j,\notag
  \end{align}
 for any $Y_n$. Take a sufficiently large $n$ so that $\eta_n<\delta_j$. Then we have 
 \begin{align}
  \Obs_Y(X;-(\kappa+\varepsilon_j))&\leq \Obs_{Y_n}^{\delta_j}(X;-(\kappa+\varepsilon_j))+5\delta_j\notag\\
  &\leq \Obs_{Y_n}(X;-\kappa)+2\varepsilon_j+5\delta_j, \notag
 \end{align} 
  where the first inequality comes from (\ref{eq:pGHlsc}). Therefore letting $n\rightarrow\infty$ leads 
  \begin{align}
   \Obs_Y(X;-(\kappa+\varepsilon_j))\leq \liminf_{n\rightarrow\infty}\Obs_{Y_n}(X;-\kappa)+2\varepsilon_j+5\delta_j.\notag
  \end{align}
  Since $j$ is arbitrary in the above inequality and $\Obs_Y(X;-\kappa)$ is continuous at $\kappa$, we obtain 
  \begin{align}
   \Obs_Y(X;-\kappa)\leq \liminf_{n\rightarrow\infty}\Obs_{Y_n}(X;-\kappa).\notag
  \end{align}
  Combining two inequalities, we have the conclusion. 
 \end{proof}
\subsection{Doubling spaces}
 In this subsection, we first show that $\wtilde{\Lip_1^{\delta}}$ and $\Lip_1^{\delta}$ are comparable for doubling spaces. 
  \begin{defn}[doubling space]
   Let $(X,d_X,\mu_X)$ be a metric space with a Borel measure $\mu_X$. $X$ is said to satisfy the \emph{doubling property} if for any $x\in X$, and any $r>0$, it holds 
   \begin{align}
    \mu_X(B_{2r}(x))\leq 2^N\mu_X(B_r(x)).\notag
   \end{align} 
   We call $2^N$ a doubling constant. 
  \end{defn}
  \begin{rem}
   Let $(X,d_X,\mu_X)$ be a metric space with a Borel measure $\mu_X$. Assume $X$ satisfies the doubling property with the doubling constant $2^N$. Then for any $x\in X$ and $0<r\leq R$, we have 
   \begin{align}
    \mu_X(B_R(x))\leq \left(\frac{2R}{r}\right)^N\mu_X(B_r(x))\notag
   \end{align} 
   by a simple argument. See for instance the proof of Theorem 5.2.2 in \cite{AT}. 
  \end{rem}
 The following Lemmata are standard. For instance, see \cite{BBI}*{p278} and \cite{Shioya}*{Lemma 3.4}, respectively for the proof. 
  \begin{lem}
   Let $X$ be a compact mm space and $\eta>0$ a positive number. Then there exists a finite maximal $\eta$-separated set $\calN$ in $X$, that is, $d_X(x,y)\geq \eta$ for any $x,y\in\calN$ and $\calN$ is the maximal, by inclusion, subset satisfying such a property. Moreover $B_{\eta}(\calN)=X$ holds. 
  \end{lem}
  \begin{lem}
   Let $X$ be an mm space and $\calN\subset X$ a Borel subset. Then there exists a Borel map $\pi:X\rightarrow \calN$ such that 
   \begin{align}
    d_X(x,\pi(x))=\min_{y\in\calN}d_X(x,y)\notag
   \end{align}
   holds. Moreover $\pi$ can be constructed satisfying $\pi(x)=x$ for any $x\in\calN$. $\pi$ is called a nearest point map. 
  \end{lem}
  \begin{prop}\label{prop:liplipcond}
   Let $X$ be a compact mm space and $Y$ a metric space. Assume $X$ is a doubling space with the doubling constant $2^N$. Then $Y$ is Lipschitz-approximated to $X$ if and only if 
   \begin{align}
    \lim_{\delta\rightarrow +0}H\!d_{KF}^{\mu_X}(\Lip_1^{\delta}(X,Y),\Lip_1(X,Y))=0.\label{eq:lipapp}
   \end{align}
  \end{prop}
  \begin{proof}
   Since $\Lip_1^{\delta}\subset \wtilde{\Lip_1^{\delta}}$ holds, only if part is trivial. Hence we prove ``if" part. Assume (\ref{eq:lipapp}). We denote by $D$ the diameter of $X$ for simplicity. Take an arbitrary $f\in \wtilde{\Lip_1^{\delta}}(X,Y)$. Then there exists a Borel set $\tilde{X}\subset X$ with $\mu_X(\tilde{X})\geq 1-\delta$ such that 
   \begin{align}
    d_Y(f(x),f(z))\leq d_X(x,z)+\delta\notag
   \end{align}
   for all $x,z\in\tilde{X}$. Let $\pi:X\rightarrow \tilde{X}$ be a nearest point map. It holds that $d_X(x,\pi(x))\leq 3D\delta^{1/N}$. Indeed, suppose $B_{\varepsilon}(x)\cap \tilde X= \emptyset$ for $\varepsilon>0$. Then by the doubling property, we have 
   \begin{align}
    1=\mu_X(X)=\mu_X(B_D(x))\leq \left(\frac{2D}{\varepsilon}\right)^N\mu_X(B_{\varepsilon}(x)).\notag
   \end{align}
   On the other hand, $\mu_X(B_{\varepsilon})\leq \mu_X(\tilde X^c)<\delta$. Combining these two inequalities, we get $\varepsilon< 2D\delta^{1/N}$. Hence $B_{3D\delta^{1/N}}(x)\cap \tilde X\neq\emptyset$ for any $x\in X$. 
   \par We define a function $g:=f\circ\pi$. Thus we have 
   \begin{align}
    d_Y(g(x),g(z))\leq d_X(\pi(x),\pi(z))+\delta\leq d_X(x,z)+6D\delta^{1/N}+\delta\notag
   \end{align}
   for any $x,z\in X$. Hence $g\in \Lip_1^{\delta+6D\delta^{1/N}}(X,Y)$ and $d_{KF}^{\mu_X}(f,g)\leq \delta$ by the construction of $g$. Therefore we obtain 
   \begin{align}
    &\lim_{\delta\rightarrow+0}H\!d_{KF}^{\mu_X}(\Lip_1(X,Y),\wtilde{\Lip_1^{\delta}}(X,Y))\notag\\
    &\leq \lim_{\delta\rightarrow+0}\left(H\!d_{KF}^{\mu_X}(\Lip_1(X,Y),\Lip_1^{\delta+6D\delta^{1/N}}(X,Y))+\delta\right)=0.\notag
   \end{align}
  \end{proof}
 \subsection{Compact Alexandrov spaces of finite dimension}
   First we give the definitions of $\CAT(0)$ and non-negatively curved Alexandrov spaces.  
  \begin{defn}[Curved spaces in the sense of Alexandrov]
   Let $X$ be a geodesic metric space. We say that $X$ is non-negatively curved in the sense of Alexandrov if the following condition is held; let $x,y,z\in X$ be arbitrary distinct three points, which are not on a single geodesic, then it holds 
   \begin{align}
    d_X^2(x,\gamma_t)\geq (1-t)d_X^2(x,y)+td_X^2(x,z)-\frac{1}{2}t(1-t)d_X^2(y,z),\notag
   \end{align}
   for any constant speed geodesic $\gamma:[0,1]\rightarrow X$ from $y$ to $z$. We call such spaces \emph{non-negatively curved Alexandrov space}. 
   \par On the other hand, $X$ is called a $\CAT(0)$ space if it holds 
   \begin{align}
    d_X^2(x,\gamma_t)\leq (1-t)d_X^2(x,y)+td_X^2(x,z)-\frac{1}{2}t(1-t)d_X^2(y,z),\notag
   \end{align} 
   where $x,y,z\in X$ and $\gamma$ satisfy the same assumptions above. 
  \end{defn} 
 See \cite{BBI,BH} for more about curved spaces. 
  \begin{lem}\label{lem:alexdoubling}
   Let $(X,d_X)$ be a non-negatively curved Alexandrov space of Hausdorff dimension $N<\infty$ and $\calH^N$ the $N$-dimensional Hausdorff measure on $(X,d_X)$. Then $(X,d_X,\calH^N)$ satisfies the doubling property with the doubling constant $2^N$.  
  \end{lem}
 The proof of Lemma can be found in, for instance, \cite{BBI}*{Theorem 10.6.6}. 
  \begin{lem}\label{lem:catlip}
   Let $(X,d_X)$ be a $\CAT(0)$ space. Take distinct three points $x,y,z\in X$ and assume these are not on a single geodesic. Let $\gamma:[0,1]\rightarrow X$, $\sigma:[0,1]\rightarrow X$ be constant speed geodesics from $x$ to $y$ and from $x$ to $z$, respectively. Then it holds 
   \begin{align}
    d_X(\gamma_t,\sigma_t)\leq td_X(y,z).\notag
   \end{align}
  \end{lem}
  The proof is easy. See for instance \cite{BH}*{p.176}. 
  \begin{rem}
   The same property holds in Busemann non-positively curved spaces. See \cite{P} for the definition and properties of Busemann spaces. 
  \end{rem}
  The following extension property of Lipschitz maps is well-known. 
  \begin{thm}[\cite{LS}]\label{thm:LS}
   Let $X$ be a non-negatively curved Alexandrov space and $Y$ a complete $\CAT(0)$ space. Let $S\subset X$ be an arbitrary subset of $X$ and $f:S\rightarrow Y$ a 1-Lipschitz map. Then there exists a 1-Lipschitz map $\bar f:X\rightarrow Y$ such that $\bar f=f$ on $S$. 
  \end{thm}
  \begin{rem}
   It is not known whether the same statement in Theorem \ref{thm:LS} for $Y$ being Busemann non-positively curved space holds or not. 
  \end{rem}
 In this subsection, we prove the following Theorem. 
  \begin{thm}\label{thm:cmmlip}
   Let $\calY$ be a family of complete $\CAT(0)$ spaces, and $(X,d_X)$ a compact non-negatively curved Alexandrov space of Hausdorff dimension $N<\infty$. We denote by $\mu_X$ the normalized $N$-dimensional Hausdorff measure on $X$, that is, $\mu_X:=\calH^N/\calH^N(X)$. $X$ is considered as an mm space $(X,d_X,\mu_X)$. Then $\calY$ is a Lipschitz-approximated family to $X$. 
  \end{thm}
  \begin{proof}
   Take an arbitrary complete $\CAT(0)$ space $Y\in\calY$. We denote by $D$ the diameter of $X$. Since $X$ is a doubling space by Lemma \ref{lem:alexdoubling}, it is sufficient to prove (\ref{eq:lipapp}) by Proposition \ref{prop:liplipcond}. Suppose (\ref{eq:lipapp}) fails. There exists a positive number $\varepsilon_0>0$ such that 
   \begin{align}
    \lim_{\delta\rightarrow+0}H\!d_{KF}^{\mu_X}(\Lip_1^{\delta}(X,Y),\Lip_1(X,Y))\geq 2\varepsilon_0.\notag
   \end{align}
   Without loss of generality, we may assume $\varepsilon_0\leq (D+4)^{-1}<1$. By the assumption, there exists a function $f\in \Lip_1^{\varepsilon_0^4}(X,Y)$ such that 
   \begin{align}
    d_{KF}^{\mu_X}(f,\Lip_1(X,Y))\geq \varepsilon_0.\label{eq:contra}
   \end{align}
   Let $\calN\subset X$ be a maximal $\varepsilon_0^2$-separated set and $\pi:X\rightarrow \calN$ a nearest point map. Take a point $x_0\in\calN$ and set $p:=f(x_0)$. A map $\sigma_{p,t}:Y\rightarrow Y$ is defined by $\sigma_{p,t}(y):=\gamma^{p,y}(t)$, where $\gamma^{p,y}:[0,1]\rightarrow Y$ is a constant speed geodesic from $p$ to $y$. By Lemma \ref{lem:catlip}, $\sigma_{p,t}$ is a $t$-Lipschitz map. 
   \par Define $t_0:=(\varepsilon_0^2+1)^{-1}$ and $g:=\sigma_{p,t_0}\circ f$. Then on $\calN$, we have 
   \begin{align}
    d_Y(g(x),g(z))&\leq t_0d_Y(f(x),f(z))\leq t_0d_X(x,z)+t_0\varepsilon_0^4\notag\\
    &=d_X(x,z)-(1-t_0)d_X(x,z)+t_0\varepsilon_0^4\notag\\
    &\leq d_X(x,z)-(1-t_0)\varepsilon_0^2+t_0\varepsilon_0^4\notag\\
    &=d_X(x,z).\notag
   \end{align}  
   Hence by Theorem \ref{thm:LS}, we are able to define a 1-Lipschitz map $h:X\rightarrow Y$ such that $g=h$ on $\calN$. Also we obtain 
   \begin{align}
    &d_Y(f(x),h(x))\leq d_Y(f(x),f(\pi(x)))+d_Y(f(\pi(x)),h(\pi(x)))+d_Y(h(\pi(x)),h(x))\notag\\
    &\leq d_X(x,\pi(x))+\varepsilon_0^4+d_Y(f(\pi(x)),g(\pi(x)))+d_X(\pi(x),x)\notag\\
    &\leq 2\varepsilon_0^2+\varepsilon_0^4+(1-t_0)d_Y(f(x_0),f(\pi(x)))\notag\\
    &\leq 3\varepsilon_0^2+\frac{\varepsilon_0^2}{\varepsilon^2_0+1}(D+\varepsilon_0^4)\leq 3\varepsilon_0^2+\varepsilon_0^2D+\varepsilon_0^2=\varepsilon_0^2(D+4)\notag\\
    &\leq \varepsilon_0\notag
   \end{align}
   for any $x\in X\setminus \calN$. Therefore $d_{KF}^{\mu_X}(f,h)\leq \varepsilon_0$. This contradicts to (\ref{eq:contra}). 
   \par By the proof above, we are able to get 
   \begin{align}
    Hd_{KF}^{\mu_X}(\Lip_1(X,Y),\Lip_1^{\delta^4}(X,Y))\leq \delta\notag
   \end{align}
   for small $\delta>0$. This estimate is independent of the choice of $Y\in\calY$. Thus $\calY$ is a Lipschitz-approximated family to $X$. 
  \end{proof}
\subsection{Finite mm spaces}
For finite mm spaces, we get a similar but stronger result. 
\begin{example}[Finite mm space]\label{ex:finite}
 Let $X$ be a finite mm space and $\frakY$ a family of complete CAT(0) spaces. Then for every $\kappa\in(0,1)$, it hold 
 \begin{align}
  \lim_{n\rightarrow \infty}\Obs_{Y_n}(X;-\kappa)=\Obs_Y(X;-\kappa)\label{eq:finCATcpt}
 \end{align}
 for compact $Y_n,Y\in\frakY$ satisfying $Y_n\xrightarrow{GH}Y$, and 
 \begin{align}
  \Obs_Y(X;-\kappa)\leq \liminf_{n\rightarrow\infty}\Obs_{Y_n}(X;-\kappa)\label{eq:finCATp}
 \end{align}
 for any $Y_n,Y\in\frakY$ satisfying $(Y_n,o_n)\xrightarrow{pGH}(Y,o)$ for points $o_n\in Y_n$, $o\in Y$. 
 \par We prove those statement as follows; let $X$ be a finite mm space and $Y$ a CAT(0) space. Set 
 \begin{align}
  d:=\min_{i\neq j}d_X(x_i,x_j)>0,\quad w:=\min_{x\in X}\mu_X(\{x\})>0.\notag
 \end{align}
 For any $\varepsilon\in(0,w)$, take $\eta\in(0,\varepsilon)$ and $\delta>0$ satisfying  
  \begin{align}
   \delta<\frac{d\varepsilon}{\di X+2\eta+2-\varepsilon}\wedge \varepsilon.\notag
  \end{align}
  Take any $f\in \wtilde{\Lip_1^{\delta+\eta}}(X,Y)$ with the non-exceptional set $\wtilde{X}$ and fix it. Then we have 
  \begin{align}
   d_Y(f(x_i), f(x_j))&\leq d_X(x_i,x_j)+\delta+\eta=d_X(x_i,x_j)\left(1+\frac{\delta}{d_X(x_i,x_j)}\right)+\eta\notag\\
   &\leq \left(1+\frac{\delta}{d}\right)d_X(x_i,x_j)+\eta\notag
  \end{align}
  for $x_i,x_j\in\wtilde{X}$. Define $p:=f(x_0)$ ($x_0\in X$). For any $y\in Y$, we define a map $\sigma_{t}(y):=\gamma_y(t)$ where $\gamma_y:[0,1]\rightarrow Y$ is a constant speed geodesic from $p$ to $y$. By Lemma \ref{lem:catlip}, it holds $y\mapsto \sigma_t(y)$ being $t$-Lipschitz. We define 
  \begin{align}
   t_{\delta}:=\left(1+\frac{\delta}{d}\right)^{-1},\quad g^{\delta}:=\sigma_{t_{\delta}}\circ f.\notag
  \end{align} 
  It holds $g^{\delta}\in\wtilde{\Lip_1^{\eta}}(X,Y)$. Indeed, for any $x_i,x_j\in \wtilde{X}$, we have 
  \begin{align}
   d_Y(g^{\delta}(x_i),g^{\delta}(x_j))&\leq t_{\delta}d_Y(f(x_i),f(x_j))\leq t_{\delta}\left(1+\frac{\delta}{d}\right)d_X(x_i,x_j)+t_{\delta}\eta\notag\\
   &\leq d_X(x_i,x_j)+\eta.\notag
  \end{align}
  It is obvious that $\mu_X(\wtilde{X})\geq 1-\delta-\eta$. Since $\delta<\varepsilon<w$, we have $\mu_X(\wtilde{X})\geq 1-\eta$. Hence $g^{\delta}\in\wtilde{\Lip_1^{\eta}}(X,Y)$. 
  \par We have 
  \begin{align}
   d_Y(f(x),g^{\delta}(x))&=(1-t_{\delta})d_Y(p,f(x))\leq (1-t_{\delta})(\di X+2(\delta+\eta))\notag\\
   &\leq (1-t_{\delta})(\di X+2(\eta+1))<\varepsilon\notag
  \end{align}
  for any $x\in\wtilde{X}$. By the definition of $\eta$, we have 
  \begin{align}
   d_P(f_{\#}\mu_X,g^{\delta}_{\#}\mu_X)\leq d^{\mu_X}_{KF}(f,g^{\delta})<\varepsilon<w.\notag
  \end{align} 
  Since $\supp g^{\delta}_{\#}\mu_X=g^{\delta}(X)$, it is sufficient to consider a Borel set $B=g^{\delta}(A)\in \calB(Y)$ with $g^{\delta}_{\#}\mu_X(B)\geq 1-\kappa$ for $A\in\calB(X)$ to estimate the partial diameter of $g^{\delta}_{\#}\mu_X$. By the definition of the Prokhorov distance, we have 
  \begin{align}
   f_{\#}\mu_X(U_{\varepsilon}(B))\geq g^{\delta}_{\#}(B)-\varepsilon\geq 1-\kappa-\varepsilon.\notag
  \end{align}
  However $\varepsilon<w$ implies $f_{\#}\mu_{X}(U_{\varepsilon}(B))\geq 1-\kappa$. Since 
  \begin{align}
   \di U_{\varepsilon}(B)\leq \di B+2\varepsilon,\notag
  \end{align}
  it holds 
  \begin{align}
   \di(f_{\#}\mu_X;1-\kappa)\leq \di B+2\varepsilon.\notag
  \end{align}
  By a similar argument in Proposition \ref{prop:arbbox}, we obtain 
  \begin{align}
   \wtilde{\Obs_Y^{\eta+\delta}}(X;-\kappa)\leq \wtilde{\Obs_Y^{\eta}}(X;-\kappa)+2\varepsilon.\notag
  \end{align}
  Therefore we have 
  \begin{align}
   \lim_{\delta\rightarrow+0}\wtilde{\Obs_Y^{\delta+\eta}}(X;-\kappa)=\wtilde{\Obs_Y^{\eta}}(X;-\kappa).\notag
  \end{align}
  By the same argument, we have 
  \begin{align}
   \wtilde{\Obs_Y^+}(X;-\kappa)=\lim_{\delta\rightarrow+0}\wtilde{\Obs_Y^{\delta}}(X;-\kappa)=\Obs_Y(X;-\kappa).\notag
  \end{align}
  It it clear that $\Lip_1^{\delta}(X,Y)=\wtilde{\Lip_1^{\delta}}(X,Y)$ for $\delta<w$. Note that the estimate is independent of the choice of $Y$. The inequality (\ref{eq:finCATcpt}), (\ref{eq:finCATp}) follows from the slightly modified argument in Theorem \ref{thm:main}. Hence we omit it. 
\end{example}
 \begin{rem}
  Along a similar argument in Example \ref{ex:finite}, we are able to prove that $\frakY$ is a Lipschitz-approximated family to each finite mm space $X$.
 \end{rem}

\section*{Acknowledgement}
The authors would like to thank Professors Takashi Shioya, Kazuhiro Kuwae, Ayato Mitsuishi, Daisuke Kazukawa, Hiroki Nakajima and Shin-ichi Ohta for their helpful comments and fruitful discussions. The first author is partly supported by JSPS KAKENHI Grant Numbers JP22K03291.   
\smallskip
\par {\bf Statements and Declarations.} No conflicts of interest or competing interests.


\begin{bibdiv}
\begin{biblist}



















\bib{AT}{book}{
   author={Ambrosio, Luigi},
   author={Tilli, Paolo},
   title={Topics on analysis in metric spaces},
   series={Oxford Lecture Series in Mathematics and its Applications},
   volume={25},
   publisher={Oxford University Press, Oxford},
   date={2004},
   pages={viii+133},
   isbn={0-19-852938-4},
   review={\MR{2039660}},
}






\bib{BBI}{book}{
   author={Burago, Dmitri},
   author={Burago, Yuri},
   author={Ivanov, Sergei},
   title={A course in metric geometry},
   series={Graduate Studies in Mathematics},
   volume={33},
   publisher={American Mathematical Society, Providence, RI},
   date={2001},
   pages={xiv+415},
   isbn={0-8218-2129-6},
   review={\MR{1835418}},
   doi={10.1090/gsm/033},
}









\bib{BH}{book}{
   author={Bridson, Martin R.},
   author={Haefliger, Andr\'{e}},
   title={Metric spaces of non-positive curvature},
   series={Grundlehren der mathematischen Wissenschaften [Fundamental
   Principles of Mathematical Sciences]},
   volume={319},
   publisher={Springer-Verlag, Berlin},
   date={1999},
   pages={xxii+643},
   isbn={3-540-64324-9},
   review={\MR{1744486}},
   doi={10.1007/978-3-662-12494-9},
}
















\bib{Fdouble}{article}{
   author={Funano, Kei},
   title={Observable concentration of mm-spaces into spaces with doubling
   measures},
   journal={Geom. Dedicata},
   volume={127},
   date={2007},
   pages={49--56},
   issn={0046-5755},
   review={\MR{2338515}},
   doi={10.1007/s10711-007-9156-6},
}



\bib{Flp}{article}{
   author={Funano, Kei},
   title={Concentration of 1-Lipschitz maps into an infinite dimensional
   $l^p$-ball with the $l^q$-distance function},
   journal={Proc. Amer. Math. Soc.},
   volume={137},
   date={2009},
   number={7},
   pages={2407--2417},
   issn={0002-9939},
   review={\MR{2495276}},
   doi={10.1090/S0002-9939-09-09873-6},
}











\bib{Grom}{book}{
   author={Gromov, Misha},
   title={Metric structures for Riemannian and non-Riemannian spaces},
   series={Modern Birkh\"{a}user Classics},
   edition={Reprint of the 2001 English edition},
   note={Based on the 1981 French original;
   With appendices by M. Katz, P. Pansu and S. Semmes;
   Translated from the French by Sean Michael Bates},
   publisher={Birkh\"{a}user Boston, Inc., Boston, MA},
   date={2007},
   pages={xx+585},
   isbn={978-0-8176-4582-3},
   isbn={0-8176-4582-9},
   review={\MR{2307192}},
}












\bib{KNS}{article}{
   author={Kazukawa, Daisuke},
   author={Nakajima, Hiroki},
   author={Shioya, Takashi}, 
   title={Principal bundle structure of the space of metric measure spaces},
   journal={arXiv:2304.06880},
}



\bib{KNS}{article}{
   author={Kazukawa, Daisuke},
   author={Nakajima, Hiroki},
   author={Shioya, Takashi},
   title={Topological aspects of the space of metric measure spaces},
   journal={Geom. Dedicata},
   volume={218},
   date={2024},
   number={3},
   pages={Paper No. 68, 28},
   issn={0046-5755},
   review={\MR{4731230}},
   doi={10.1007/s10711-024-00921-3},
}








\bib{LS}{article}{
   author={Lang, U.},
   author={Schroeder, V.},
   title={Kirszbraun's theorem and metric spaces of bounded curvature},
   journal={Geom. Funct. Anal.},
   volume={7},
   date={1997},
   number={3},
   pages={535--560},
   issn={1016-443X},
   review={\MR{1466337}},
   doi={10.1007/s000390050018},
}








\bib{Nbox}{article}{
   author={Nakajima, Hiroki},
   title={Box distance and observable distance via optimal transport}, 
   journal={arXiv:2204.04893},
}



\bib{OSh}{article}{
   author={Ozawa, Ryunosuke},
   author={Shioya, Takashi},
   title={Limit formulas for metric measure invariants and phase transition
   property},
   journal={Math. Z.},
   volume={280},
   date={2015},
   number={3-4},
   pages={759--782},
   issn={0025-5874},
   review={\MR{3369350}},
   doi={10.1007/s00209-015-1447-2},
}

\bib{P}{book}{
   author={Papadopoulos, Athanase},
   title={Metric spaces, convexity and non-positive curvature},
   series={IRMA Lectures in Mathematics and Theoretical Physics},
   volume={6},
   edition={2},
   publisher={European Mathematical Society (EMS), Z\"{u}rich},
   date={2014},
   pages={xii+309},
   isbn={978-3-03719-132-3},
   review={\MR{3156529}},
   doi={10.4171/132},
}



\bib{Shioya}{book}{
   author={Shioya, Takashi},
   title={Metric measure geometry},
   series={IRMA Lectures in Mathematics and Theoretical Physics},
   volume={25},
   note={Gromov's theory of convergence and concentration of metrics and
   measures},
   publisher={EMS Publishing House, Z\"{u}rich},
   date={2016},
   pages={xi+182},
   isbn={978-3-03719-158-3},
   review={\MR{3445278}},
   doi={10.4171/158},
}















\end{biblist}
\end{bibdiv}
\end{document}